\documentclass[12pt]{amsart}
\usepackage{amssymb,verbatim,amscd,amsmath,graphicx}
\usepackage{graphicx}
\usepackage{caption}
\usepackage{subcaption}
\usepackage{pdfsync}
\usepackage{fullpage}
\usepackage{bbm}
\usepackage[pagebackref]{hyperref}

\numberwithin{equation}{section}
\newtheorem{theorem}{Theorem}[section]
\newtheorem{lemma}[theorem]{Lemma}

\newtheorem{conjecture}[theorem]{Conjecture}

\theoremstyle{definition}

\newtheorem{def-prop}[theorem]{Definition-Proposition}

\newtheorem{remark}[theorem]{Remark}
\newtheorem{example}[theorem]{Example}

\newtheorem*{acknowledgement}{Acknowledgement}

\DeclareMathOperator{\Ass}{Ass}

\DeclareMathOperator{\Proj}{Proj}

\newcommand{\PP}{{\mathbb P}}

\newcommand{\NN}{{\mathbb N}}
\newcommand{\QQ}{{\mathbb Q}}

\newcommand{\XX}{{\mathbb X}}

\newcommand{\kk}{{\mathbbm k}}

\newcommand{\remain}[2]{\overline{#1}^{#2}}

\def\pp{{\frak p}}

\def\x{{\bf x}}
\def\y{{\bf y}}

\def\1{{\bf 1}}
\def\0{{\bf 0}}

\def\rhoa{{\rho_a}}
\def\rholim{\rho_a^{\limsup}}
\def\irho{\overline{\rho}}
\def\irhoa{\overline{\rho}_a}
\def\irholim{\overline{\rho}_a^{\limsup}}

\begin{document}
	
\title{Resurgence numbers of fiber products of projective schemes}

\author{Sankhaneel Bisui}
\address{Tulane University \\ Department of Mathematics \\
	6823 St. Charles Ave. \\ New Orleans, LA 70118, USA}
\email{sbisui@tulane.edu}
%\urladdr{http://www.math.tulane.edu/$\sim$tai/}

\author{Huy T\`ai H\`a}
\address{Tulane University \\ Department of Mathematics \\
	6823 St. Charles Ave. \\ New Orleans, LA 70118, USA}
\email{tha@tulane.edu}
%\urladdr{???}

\author{A.V. Jayanthan}
\address{Department of Mathematics, Indian Institute of Technology Madras \\
Chennai, Tamil Nadu, India - 600036.}
\email{jayanav@iitm.ac.in}

\author{Abu Chackalamannil Thomas}
\address{Tulane University \\ Department of Mathematics \\
	6823 St. Charles Ave. \\ New Orleans, LA 70118, USA}
\email{athoma17@tulane.edu}
%\urladdr{???}


\begin{abstract}
We investigate the resurgence and asymptotic resurgence numbers of fiber products of projective schemes. Particularly, we show that while the asymptotic resurgence number of the $k$-fold fiber product of a projective scheme remains unchanged, its resurgence number could strictly increase.
\end{abstract}

\maketitle

%%%%%%%%%%%%%%%%%%%%%%%%%%%%%%%%%%

\section{Introduction}

Inspired by the well-celebrated result of Ein-Lazarsfeld-Smith \cite{ELS} and Hochster-Huneke \cite{HH}, and driven by a series of conjectures and questions due to Harbourne-Huneke \cite{HH-2011}, containments between symbolic and ordinary powers of ideals have evolved to be a highly active research topic in recent years (cf. \cite{A-2017, BCH-2014, BH-2007, CEHH-2017, CGMLLPS-2016, DJ-2013, DSTG-2013, DTG-2017, KKM-2017, LM-2015, S-2015, SS-2017}). The resurgence number (defined by Bocci-Harbourne \cite{BH-2007}) and the asymptotic resurgence number (defined by Guardo-Harbourne-Van Tuyl \cite{GHV-2012}) are measures of the non-containments between these powers of ideals. Specifically, for an ideal $I$ in a polynomial ring, the \emph{resurgence number} and the \emph{asymptotic resurgence number} of $I$ are given by
$$ \rho(I) = \sup\Big\{\dfrac{m}{r} ~\Big|~ I^{(m)} \not\subseteq I^r\Big\} \text{ and }
\rhoa(I) = \sup\Big\{\dfrac{m}{r} ~\Big|~ I^{(mt)} \not\subseteq I^{rt}, t \gg 0\Big\}.$$

It is easy to see that $\rhoa(I) \le \rho(I)$ for any ideal $I$. However, a priori, it is not quite clear how different these invariants could be.  In fact, if one replaces the ordinary power of $I$ by its integral closure and defines
$$\irho(I)  = \sup\Big\{\dfrac{m}{r} ~\Big|~ I^{(m)} \not\subseteq \overline{I^r}\Big\} \text{ and }
\irhoa(I) = \sup\Big\{\dfrac{m}{r} ~\Big|~ I^{(mt)} \not\subseteq \overline{I^{rt}}, t \gg 0\Big\},$$
then the main result of a recent work of Dipasquale-Francisco-Mermin-Schweig \cite{DFMS-2018} shows that
$$\irho(I) = \irhoa(I) = \rhoa(I).$$
Moreover, since these invariants are hard to compute, there are very few examples where $\rho(I)$ and $\rhoa(I)$ are known explicitly (cf. \cite{DHNSSTG-2015}).
Our goal in this paper is to study the resurgence and asymptotic resurgence numbers of fiber products of projective schemes. Particularly, we exhibit a pathological example of the difference between these two invariants, and provide a large family of ideals for which the asymptotic resurgence number could be computed explicitly.

Our results show that while the asymptotic resurgence of a fiber product of projective schemes can be computed via that of given schemes, it is not necessarily the case for the resurgence number. To be more specific, let $A = \kk[\PP^N_\kk]$, $B = \kk[\PP^M_\kk]$, $X = \Proj A/I$, and $Y = \Proj B/J$, for homogeneous ideals $I \subseteq A$ and $J \subseteq B$. Then, the fiber product $X \times_\kk Y$, embedded in $\PP^N_\kk \times_\kk \PP^M_\kk$, is defined by ideal $I+J \subseteq A \otimes_\kk B$. We prove the following theorems.

\medskip

\noindent{\bf Theorems \ref{thm.rhoa} and \ref{thm.rho}.} Let $I \subseteq A$ and $J \subseteq B$ be nonzero proper homogeneous ideals. Let $I+J$ denote the sum of extensions of $I$ and $J$ in $A \otimes_\kk B$. Then
\begin{enumerate}
	\item $\rhoa(I+J) = \max \{\rhoa(I), \rhoa(J)\}.$
	\item $\max\{\rho(I), \rho(J)\} \le \rho(I+J) \le \rho(I)+\rho(J).$
\end{enumerate}

The bounds in Theorem \ref{thm.rho} may be strict. Particularly, we show that by taking $k$-fold fiber products of a projective scheme, as the asymptotic resurgence number remains the same, the resurgence number could strictly increase.

\medskip

\noindent{\bf Theorem \ref{thm.example}.} There exists a set of points $\XX \subseteq \PP^2_\kk$ such that if $I_\XX^{[k]}$ represents the defining ideal of the $k$-fold fiber product of $\XX$, embedded in the multiprojective space $\PP^2_\kk \times_\kk \dots \times_\kk \PP^2_\kk$ ($k$ times), then
$$\rhoa(I_\XX^{[k]}) = \rhoa(I_\XX) \text{ and } \rho(I_\XX^{[k]}) > \rho(I_\XX).$$

We shall now briefly describe our methods in proving these results. To prove Theorem \ref{thm.rhoa}, we establish the following statements (see Section \ref{sec.resurgence} for the definition of $\rholim(\bullet)$ and $\irholim(\bullet)$):
\begin{enumerate}
\item $\rhoa(I+J) \le \max\{\rholim(I), \rholim(J)\} \le \rholim(I+J)$;
%\item $\max\{\rholim(I), \rholim(J)\} \le \rholim(I+J)$;
\item $\rholim(I) = \irholim(I) = \rhoa(I)$.
\end{enumerate}
Statement (1) is achieved by a careful analysis of containments between powers of $I$, $J$ and $(I+J)$, and invoking the binomial expansion given in \cite{HNTT-2017}, which states that
$(I+J)^{(h)} = \sum_{k=0}^h I^{(k)}J^{(h-k)}.$
Statement (2) is proved by employing techniques in \cite{DFMS-2018}, in order to show that $\rholim(I) = \irholim(I)$, and making use of the main result of \cite{DFMS-2018}, which gives $\irholim(I) = \rhoa(I)$.

Theorem \ref{thm.rho} is established using a similar line of arguments as that of statement (1). Finally, Theorem \ref{thm.example} arises as a consequence of Theorem \ref{thm.kfold}, which derives a non-containment for powers of $I_\XX^{[k]}$ based on a given non-containment for powers of $I_\XX$. This result is achieved by an interesting Gr\"obner basis argument.

We assume that the reader is familiar with basic constructions in commutative algebra and algebraic geometry. For unexplained terminology, we refer the reader to standard texts in the research area \cite{Eisenbud, Hartshorne}.

\begin{acknowledgement} The authors thank Louiza Fouli and Paolo Mantero for pointing out a mistake in the first draft of the paper. The authors also thank Elena Guardo and Alexandra Seceleanu for many stimulating discussions on related topics. Part of this work was done while the third author visited the other authors at Tulane University --- the authors thank Tulane University for its hospitality. The second author is partially supported by Louisiana Board of Regents (grant \#LEQSF(2017-19)-ENH-TR-25). Some of the computation in this paper was done using Macaulay 2 \cite{M2}.
\end{acknowledgement}

%%%%%%%%%%%%%%%%%%%%%%%%%%%%%%%%%%

\section{Resurgence and asymptotic resurgence numbers} \label{sec.resurgence}

In this section, we investigate the resurgence and asymptotic resurgence numbers of fiber products of projective schemes. We shall start by recalling relevant definitions and a result of \cite{DFMS-2018} that we will use.

Recall that for an ideal $I$ in a commutative ring $R$ and $m \in \NN$, the \emph{$m$-th symbolic power} of $I$ is defined to be
$$I^{(m)} = \bigcap_{\pp \in \Ass(I)} \left(I^mR_\pp \cap R\right).$$
Recall also that for an ideal $I$, the \emph{resurgence number} and \emph{asymptotic resurgence number} of $I$ are defined to be
$$ \rho(I) = \sup\Big\{\dfrac{m}{r} ~\Big|~ I^{(m)} \not\subseteq I^r\Big\} \text{ and }
\rhoa(I) = \sup\Big\{\dfrac{m}{r} ~\Big|~ I^{(mt)} \not\subseteq I^{rt}, t \gg 0\Big\}.$$
Related to these resurgence numbers, we further have the following invariants
\begin{align*}
\rho(I,t) & = \sup\Big\{\dfrac{m}{r} ~\Big|~ I^{(m)} \not\subseteq I^r, m \ge t, r \ge t\Big\},\\
\rholim(I) & = \limsup_{t \rightarrow \infty} \rho(I,t).
\end{align*}
It is easy to see from the definition that $\rhoa(I) \le \rholim(I) \le \rho(I)$ for any ideal $I$.

In a recent work \cite{DFMS-2018}, DiPasquale, Francisco, Mermin and Schweig introduced similar invariants replacing the ordinary powers with their integral closures. Particularly, they define
\begin{align*}
\irho(I) & = \sup\Big\{\dfrac{m}{r} ~\Big|~ I^{(m)} \not\subseteq \overline{I^r}\Big\}, \\
\irhoa(I) & = \sup\Big\{\dfrac{m}{r} ~\Big|~ I^{(mt)} \not\subseteq \overline{I^{rt}}, t \gg 0\Big\}, \\
\irho(I,t) & = \sup\Big\{\dfrac{m}{r} ~\Big|~ I^{(m)} \not\subseteq \overline{I^r}, m \ge t, r \ge t\Big\},\\
\irholim(I) & = \limsup_{t \rightarrow \infty} \irho(I,t).
\end{align*}
By definition, we have $\irhoa(I) \le \irholim(I) \le \irho(I)$. It also follows from the definition that
$$\irhoa(I) \le \rhoa(I), \irholim(I) \le \rholim(I) \text{ and } \irho(I) \le \rho(I).$$
We shall make use of the following interesting connection between these invariants.

\begin{theorem}[\protect{\cite[Proposition 4.2]{DFMS-2018}}] \label{thm.DFMS}
	Let $I$ be any ideal in a polynomial ring. Then
	$$\rhoa(I) = \irhoa(I) = \irholim(I) = \irho(I).$$
\end{theorem}

Our first lemma establishes the equality between $\rholim(I)$ and its integral closure version. This equality is in the same spirit as that of for $\rhoa(I)$ given in \cite[Proposition 4.2]{DFMS-2018} (see Theorem \ref{thm.DFMS}).

\begin{lemma} \label{lem.rhoalim}
	Let $I$ be any ideal in a polynomial ring. Then, we have
	$$\rholim(I) = \irholim(I).$$
\end{lemma}

\begin{proof} It is clear that if $I^{(h)} \not\subseteq \overline{I^r}$, then $I^{(h)} \not\subseteq I^r$. Thus, for each $t \in \NN$, we have $\irho(I,t) \le \rho(I,t)$. This implies that $\irholim(I) \le \rholim(I)$. It remains to prove that
	$$\rholim(I) \le \irholim(I).$$
Consider any $\theta \in \QQ$ such that $\theta > \irholim(I)$. It suffices to show that $\rholim(I) \le \theta$.

It is known that the integral closure of the Rees algebra of $I$ is finitely generated over the Rees algebra of $I$ (see, for example, \cite[Proposition 5.3.4]{HS-06}). Thus, there exists an integer $k$ such that $\overline{I^r} = I^{r-k}\overline{I^k} \subseteq I^{r-k}$ for all $r \ge k$.

Since $\theta > \limsup_{t \rightarrow \infty} \rho(I,t)$, there exists $t_0 \in \NN$ such that $\theta > \irho(I,t)$ for all $t \ge t_0$. Set $\epsilon = \theta - \irho(I,t_0) > 0$. Choose $r_0 \in \NN$ such that $\dfrac{k\theta}{r_0} < \epsilon$. Then, for any $h, r \in \NN$ such that $h, r \ge \max\{t_0,r_0+k\}$ and $\dfrac{h}{r} = \theta$, we have $\dfrac{h}{r+k} > \dfrac{h}{r} - \epsilon > \irho(I,t_0)$. This implies that $I^{(h)} \subseteq \overline{I^{r+k}} \subseteq I^r$.  Therefore, $\rho(I,t) \le \theta$ for all $t \ge \max\{t_0, r_0+k\}$. We conclude that $\rholim(I) \le \theta$, and the assertion is proved.
\end{proof}

Before proceeding, let us fix a number of notations for the rest of the paper. Let $\kk$ denote a field, let $\PP^N_\kk$ be the $N$-dimensional projective space over $\kk$, and let $\kk[\PP^N_\kk]$ represent its corresponding polynomial ring. For fixed positive integers $N$ and $M$, let $A = \kk[\PP^N_\kk]$ and $B = \kk[\PP^M_\kk]$. Let $I \subseteq A$ and $J \subseteq B$ be nonzero proper homogeneous ideals, and let $X = \Proj A/I$ and $Y = \Proj B/J$. It is a basic fact that
$$X \times_\kk Y = \text{Bi-Proj } R/(I+J) \subseteq \PP^N_\kk \times_\kk \PP^M_\kk,$$
where $R = A \otimes_\kk B$, and $I+J$ represents the sum of extensions of $I$ and $J$ in $R$.

The next few lemmas are essential in the computation of $\rhoa(I+J)$. For simplicity of notation, we shall use $\x$ and $\y$ to represent the coordinates of $\PP^N_\kk$ and $\PP^M_\kk$, respectively.

\begin{lemma} \label{lem.rhoalimmax}
	$\rholim(I+J) \ge \max\{\rholim(I), \rholim(J)\}.$
\end{lemma}

\begin{proof} Consider any $\theta \in \QQ$ such that $\theta > \rholim(I+J)$. By definition, there eixsts $t_0 \in \NN$ such that for all $t \ge t_0$, $\theta > \rho(I+J,t)$. Thus, for any $h,r \in \NN$ such that $h,r \ge t \ge t_0$ and $\dfrac{h}{r} \ge \theta$, we must have $(I+J)^r \supseteq (I+J)^{(h)} \supseteq I^{(h)}$.
	
Let $f(\x)$ be a minimal generator of $I^{(h)} \subseteq A$. Then, $f(\x) \in (I+J)^r = \sum_{i=0}^r I^iJ^{r-i}$. Thus, we can write $f(\x) = \sum_{i=0}^r f_i(\x,\y)$, where $f_i(\x,\y) \in I^iJ^{r-i}$.
Observe that for each $i < r$, every term of every element in $J^{r-i}$ must contain nontrivial powers of the variables $\y$. Therefore, the same is true also for every element in $I^iJ^{r-i}$. This, since $f(\x)$ contains no terms involving the variables $\y$, implies that the $f_i(\x,\y)$, for $i < r$, must cancel leaving $f(\x) = f_r(\x,\y) \in I^r$. That is, $I^{(h)} \subseteq I^r$.

We have shown that $\theta > \rho(I,t)$ for all $t \ge t_0$. Hence, $\theta > \rholim(I)$. Similarly, we also have $\theta > \rho(J,t)$ for all $t \ge t_0$ and, thus, $\theta > \rholim(J)$. The conclusion now follows since this is true for any rational number $\theta > \rholim(I+J)$.
\end{proof}

\begin{lemma} \label{lem.rhoamax}
	$\rhoa(I+J) \le \max\{\rholim(I), \rholim(J)\}.$
\end{lemma}

\begin{proof} It suffices to show that for any $\theta \in \QQ$ such that $\theta > \max\{\rholim(I), \rholim(J)\}$, we have $\rhoa(I+J) \le \theta$. Set $\epsilon = \theta - \max\{\rholim(I), \rholim(J)\} > 0$.
	
By definition, there exists $q_0 \in \NN$ such that for any $p,q \in \NN$, $q \ge q_0$ and $\dfrac{p}{q} \ge \theta$, we must have $I^{(p)} \subseteq I^q$ and $J^{(p)} \subseteq J^q$. Choose $t_0 \in \NN$ such that $\dfrac{q_0\theta}{t_0} < \epsilon$. We shall prove that for any $h,r \in \NN$ such that $\dfrac{h}{r} > \theta$ and $t \gg 0$, we must have $(I+J)^{(ht)} \subseteq (I+J)^{rt}$, which then implies that $\rhoa(I+J) \le \theta$.

By \cite[Theorem 3.4]{HNTT-2017}, we have
$$(I+J)^{(ht)} = \sum_{i = 0}^{ht} I^{(i)}J^{(ht-i)}.$$
Thus, it is enough to show that for any $0 \le i \le ht$, the containment $I^{(i)} J^{(ht-i)} \subseteq (I+J)^{rt}$ holds.

Indeed, by choosing $t \ge \max\left\{t_0, \dfrac{\theta}{h-r\theta}\right\}$, we have $\theta - \dfrac{q_0\theta}{t} > \max \{\rholim(I), \rholim(J)\}$ and $\dfrac{ht}{\theta} - rt > 1$. If $i < q_0\theta$, then $\dfrac{ht-i}{rt} = \dfrac{h}{r} - \dfrac{i}{rt} > \dfrac{h}{r} - \dfrac{q_0\theta}{t} > \max \{\rholim(I), \rholim(J)\}$. Thus, $J^{(ht-i)} \subseteq J^{rt}$ for all $t \gg 0$. This implies that $I^{(i)}J^{(ht-i)} \subseteq J^{(ht-i)} \subseteq J^{rt} \subseteq (I+J)^{rt}$.

If $i > rt\theta - q_0\theta$, then $\dfrac{i}{rt} > \theta - \dfrac{q_0\theta}{rt} \ge \theta - \dfrac{q_0\theta}{t} > \max\{\rholim(I), \rholim(J)\}$. Thus, $I^{(i)} \subseteq I^{rt}$ for $t \gg 0$. This also implies that $I^{(i)}J^{(ht-i)} \subseteq I^{rt} \subseteq (I+J)^{rt}$.

Suppose now that $q_0\theta \le i \le rt\theta - q_0\theta$. Let $i' = \big\lfloor \dfrac{i}{\theta} \big\rfloor \ge q_0$. Then, $I^{(i)} \subseteq I^{i'}$, since $\dfrac{i}{i'} \ge \theta$ and $i,i' \ge q_0$. Also, since $\dfrac{ht}{\theta} - rt > 1$, we have $rt - \dfrac{ht-i}{\theta} < i' \le \dfrac{i}{\theta}$.  This implies that $\dfrac{ht-i}{rt-i'} \ge \theta$. Moreover, since $i \le (rt-q_0)\theta$, we have $q_0 \le rt - \dfrac{i}{\theta} \le rt - i'$. Therefore, $J^{(ht-i)} \subseteq J^{rt-i'}$. Hence,
$$I^{(i)}J^{(ht-i)} \subseteq I^{i'}J^{rt-i'} \subseteq (I+J)^{rt}.$$
\end{proof}

We are now ready to state our main results of the section, which compute $\rhoa(I+J)$ and give bounds for $\rho(I+J)$.

\begin{theorem} \label{thm.rhoa}
	Let $I \subseteq A$ and $J \subseteq B$ be nonzero proper homogeneous ideals. Let $I+J$ be the sum of extensions of $I$ and $J$ in $R = A \otimes_\kk B$. Then, we have
	$$\rhoa(I+J) = \irhoa(I+J) = \irho(I+J) = \max\{ \rhoa(I), \rhoa(J)\}.$$
\end{theorem}

\begin{proof} It follows from Theorem \ref{thm.DFMS} that $\rhoa(I+J) = \irhoa(I+J) = \irho(I+J)$. It remains to establish the last equality. By definition, $\irhoa(I+J) \le \irholim(I+J) \le \irho(I+J)$. Thus, together with Lemmas \ref{lem.rhoalim}, \ref{lem.rhoalimmax} and \ref{lem.rhoamax}, we get
\begin{align}
\rhoa(I+J) & = \irhoa(I+J) = \irho(I+J) = \irholim(I+J) = \rholim(I+J) \label{eq.rhoa1}\\
& \ge \max\{\rholim(I), \rholim(J)\}  \nonumber \\
& \ge \rhoa(I+J). \nonumber
\end{align}
Hence, we must have equalities in (\ref{eq.rhoa1}) and the assertion follows.
\end{proof}

\begin{theorem} \label{thm.rho}
	Assume the same hypotheses as in Theorem \ref{thm.rhoa}. Then,
	$$\max \{\rho(J), \rho(J)\} \le \rho(I+J) \le \rho(I) + \rho(J).$$
\end{theorem}

\begin{proof} Consider any $\theta \in \QQ$ such that $\theta > \rho(I+J)$. Then, by definition, for any $h, r \in \NN$ such that $\dfrac{h}{r} \ge \theta$, we have $(I+J)^r \supseteq (I+J)^{(h)} \supseteq I^{(h)}$. By the same line of arguments as that of Lemma \ref{lem.rhoalimmax}, it then can be shown that $I^r \supseteq I^{(h)}$. This implies that $\rho(I) \le \theta$. Similarly, we have $\rho(J) \le \theta$. Thus, $\theta \ge \max\{\rho(I), \rho(J)\}$. Since this is true for any $\theta > \rho(I+J)$, we deduce that $\rho(I+J) \ge \max\{\rho(I), \rho(J)\}$.
	
We shall now prove the second inequality. Consider any $h,r \in \NN$ such that $\dfrac{h}{r} > \rho(I)+\rho(J)$. It suffices to show that $(I+J)^{(h)} \subseteq (I+J)^r$. Indeed, by \cite[Theorem 3.4]{HNTT-2017}, we have
$$(I+J)^{(h)} = \sum_{i=0}^h I^{(i)}J^{(h-i)}.$$
Thus, it suffices to prove that for any $0 \le i \le h$, $I^{(i)}J^{(h-i)} \subseteq (I+J)^r$.

To this end, observe that if $i \le r\rho(I)$, then $h-i > r\rho(J)$, and so $J^{(h-i)} \subseteq J^r \subseteq (I+J)^r$, which implies that $I^{(i)}J^{(h-i)} \subseteq J^{(h-i)} \subseteq (I+J)^r$. On the other hand, if $i > r\rho(I)$, then $I^{(i)} \subseteq I^r \subseteq (I+J)^r$, which also implies that $I^{(i)}J^{(h-i)} \subseteq I^{(i)} \subseteq (I+J)^r$.
\end{proof}

\begin{example}
Let $p$ be an odd prime. Let $\kk$ be a field of characteristic $p$ consisting of $q= p^t$ elements, for some $t \in \NN$, and let $\kk'$ be the prime subfield of $\kk$. Let $I \subseteq \kk[\PP^N_\kk]$ and $J \subseteq \kk[\PP^M_\kk]$ be the defining ideals of all but one $\kk'$-points in $\PP^N_\kk$ and $\PP^M_\kk$, respectively, and let $I+J$ denote the sum of their extensions in $\kk[\PP^N_\kk \times_\kk \PP^M_\kk]$. Then, it follows from \cite[Theorem 3.2]{DHNSSTG-2015}, and Theorems \ref{thm.rhoa} and \ref{thm.rho} that
$$\rhoa(I+J) = \dfrac{\max\{N,M\}(q-1)+1}{q},$$
$$\dfrac{\max\{N,M\}(q-1)+1}{q} \le \rho(I+J) \le \dfrac{(N+M)(q-1)+2}{q}.$$
\end{example}

%%%%%%%%%%%%%%%%%%%%%%%%%%%%%%%%%%%%%%%%%

\section{$k$-fold fiber products of a projective scheme} \label{sec.kfold}

In this section, we shall focus on $k$-fold fiber products of a projective scheme, and give an example where the asymptotic resurgence number stays the same while the resurgence number strictly increases. The next few lemmas derive a non-containment for $I+J$ from that of $I$ and $J$.

Recall that we use $\x$ and $\y$ to denote the variables in $A$ and $B$. For term orders $\prec_1$ on $A$ and $\prec_2$ on $B$, we define $\prec_{1,2}$ to be the term order on $\kk[\x,\y] = R = A \otimes_\kk B$ obtained by combining $\prec_1$ and $\prec_2$ such that $x_i > y_j$ for all $i,j$.

\begin{lemma} \label{lem.Gbasis}
Let $Q \subseteq A$ and $H \subseteq B$ be ideals. Suppose that $G$ and $G'$ are Gr\"obner bases for $Q$ and $H$ with respect to some term orders $\prec_1$ and $\prec_2$, respectively. Then $G \cup G'$ is a Gr\"obner basis for $Q+H \subseteq R$ with respect to the term order $\prec_{1,2}$.
\end{lemma}

\begin{proof} It suffices to show that for any $a(\x) \in G$ and $b(\y) \in G'$, the S-polynomial $S(a(\x),b(\y))$ reduces to 0 modulo $G \cup G'$ with respect to $\prec_{1,2}$. Indeed, we may assume that $a(\x)$ and $b(\y)$ are monic polynomials, and write $a(\x) = \x^\alpha + a'(\x)$ and $b(\y) = \y^\beta + b'(\y)$, where $\x^\alpha$ and $\y^\beta$ are leading terms of $a(\x)$ and $b(\y)$ with respect to $\prec_1$ and $\prec_2$. Then,
	$$S(a(\x),b(\y)) = \y^\beta a'(\x) - \x^\alpha b'(\y) = a(\x)b'(\y) - b(\y)a'(\x)$$
which clearly reduces to 0 module $G \cup G'$. The assertion is proved.
\end{proof}

Let $G$ be a Gr\"obner basis for an ideal $I$, we shall write $\remain{f}{G}$ for the remainder of $f$ modulo $G$.

\begin{lemma} \label{lem.fg}
	Let $I \subseteq A$ and $J \subseteq B$ be ideals. Suppose that $f(\x) \in A$ and $g(\y) \in B$ are polynomials such that $f(\x) \not\in I^r$ and $g(\y) \not\in J^s$. Then, $f(\x)g(\y) \not\in (I+J)^{r+s-1}$.
\end{lemma}

\begin{proof} Let $I^r = \bigcap_{i=1}^u Q_i$ and $J^s = \bigcap_{j=1}^v H_j$ be primary decompositions of $I^r$ and $J^s$, respectively. It is easy to see that $(I+J)^{r+s-1} \subseteq I^r + J^s \subseteq \bigcap_{i,j} (Q_i + H_j)$. Thus, if $f(\x) g(\y) \in (I+J)^{r+s-1}$, then we must have $f(\x)g(\y) \in Q_i + H_j$ for all $1 \le i \le u$ and $1 \le j \le v$. We shall show that this forces either $f(\x) \in Q_i$ for all $i$ or $g(\y) \in H_j$ for all $j$, which then results in either $f(\x) \in I^r$ or $g(\y) \in J^s$, a contradiction.
	
Indeed, fix an $i$ and $j$. Let $G$ and $G'$ be Gr\"obner bases for $Q_i$ and $H_j$, respectively, with respect to some term orders $\prec_1$ in $A$ and $\prec_2$ in $B$. By Lemma \ref{lem.Gbasis}, $G \cup G'$ is a Gr\"obner basis for $Q_i+H_j$ with respect to $\prec_{1,2}$. Since $f(\x)g(\y) \in Q_i+H_j$, we have $\remain{f(\x)g(\y)}{G \cup G'} = 0$.
Observe, from the definition of $\prec_{1,2}$, that
$$\remain{f(\x)g(\y)}{G \cup G'} = \remain{\remain{f(\x)}{G} g(\y)}{G'} =  \remain{f(\x)}{G}\remain{g(\y)}{G'}.$$
Thus, we have either $\remain{f(\x)}{G} = 0$ or $\remain{g(\y)}{G'} = 0$. That is, either $f(\x) \in Q_i$ or $g(\y) \in H_j$.

Since this is true for any $i$ and $j$, either $f(\x) \in Q_i$ for all $i$ or, if there exists an $i$ such that $f(\x) \not\in Q_i$, then that forces $g(\y) \in H_j$ for all $j$. The assertion is proved.
\end{proof}

\begin{lemma} \label{lem.IJ}
	Let $I \subseteq A$ and $J \subseteq B$ be nonzero proper homogeneous ideals. Suppose that $I^{(h)} \not\subseteq I^r$ and $J^{(k)} \not\subseteq J^s$. Then,
	$$(I+J)^{(h+k)} \not\subseteq (I+J)^{r+s-1}.$$
\end{lemma}

\begin{proof} Let $f(\x) \in I^{(h)} \setminus I^r$ and $g(\y) \in J^{(k)} \setminus J^s$. By Lemma \ref{lem.fg}, we have $f(\x)g(\y) \not\in (I+J)^{r+s-1}$. On the other hand, by \cite[Theorem 3.4]{HNTT-2017}, we have
	$$f(\x)g(\y) \in I^{(h)}J^{(k)} \subseteq (I+J)^{(h+k)}.$$
Hence, $(I+J)^{(h+k)} \not\subseteq (I+J)^{r+s-1}$.
\end{proof}

\begin{remark} In the remaining of the paper, for a homogeneous ideal $I \subseteq A$ and $X = \Proj A/I$, we shall
\begin{enumerate}
\item use $X^{[k]}$ to denote the $k$-fold fiber product $\underbrace{X \times_\kk \dots \times_\kk X}_{k \text{ times}}$, and
\item use $I^{[k]}$ to denote the defining ideal of $X^{[k]}$ embedded in $\underbrace{\PP^N_\kk \times_\kk \dots \times_\kk \PP^N_\kk}_{k \text{ times}}$.
\end{enumerate}
\end{remark}

Our next result provides a lower bound for the resurgence number of $k$-fold fiber products.

\begin{theorem} \label{thm.kfold}
	Let $I \subseteq A$ be a nonzero proper homogeneous ideal. Suppose that $I^{(h)} \not\subseteq I^r$ for some $h,r \in \NN$. Then, for any $k \in \NN$, we have
	$$\rho(I^{[k]}) \ge \dfrac{kh}{k(r-1)+1}.$$
\end{theorem}

\begin{proof} It suffices to show that $\big(I^{[k]}\big)^{(kh)} \not\subseteq \big(I^{[k]}\big)^{k(r-1)+1}.$ Indeed, this non-containment follows inductively by applying Lemma \ref{lem.IJ} with $J = I^{[k-1]}$.
\end{proof}

\begin{example} Let $p$ be an odd prime. Let $\kk$ be a field of characteristic $p$ and let $\kk'$ be its prime subfield. Let $N = \dfrac{p+1}{2}$ and let $I$ be the defining ideal of all but one $\kk'$-points in $\PP^N_\kk$. It follows from \cite[Theorem 3.9]{HS-2015} that
$$I^{(N+1)} \not\subseteq I^2.$$

By applying Theorem \ref{thm.kfold} we get, for any $k \ge 2$,
$$\rho(I^{[k]}) \ge \dfrac{k(N+1)}{k+1}.$$
Observe that $\lim_{k \rightarrow \infty} \dfrac{k(N+1)}{k+1} = N+1$. Thus, by taking $p \rightarrow \infty$, we get a family of ideals $I$ (depending on $p$) such that $\rho(I^{[k]})$ gets arbitrarily large as $k \rightarrow \infty$.
\end{example}

As a consequence of Theorem \ref{thm.kfold}, we give an example where $\rhoa(I^{[k]})$ remains unchanged while $\rho(I^{[k]})$ strictly increases.

\begin{theorem} \label{thm.example}
There exists a set of points $\XX \subseteq \PP^2_\kk$ such that if $I_\XX^{[k]}$ represents the defining ideal of the $k$-fold fiber product of $\XX$, embedded in the multiprojective space $\PP^2_\kk \times_\kk \dots \times_\kk \PP^2_\kk$ ($k$ times), then
$$\rhoa(I_\XX^{[k]}) = \rhoa(I_\XX) \text{ and } \rho(I_\XX^{[k]}) > \rho(I_\XX).$$
\end{theorem}

\begin{proof} Let $n \in \NN$ and let $\kk$ be a field of characteristic not equal to 2 containing $n$ distinct roots of unity. Let $A = \kk[x,y,z]$ and let
$$I = (x(y^n-z^n), y(z^n-x^n), z(x^n-y^n)).$$
Then $I = I_\XX$ is the defining ideal of a set $\XX$ of $n^2+3$ points in $\PP^2_\kk$ (see \cite{HS-2015}). It is known from \cite[Theorem 2.1]{DHNSSTG-2015} that $\rhoa(I) = (n+1)/n$. It is also known from \cite[Theorem 2.1]{DHNSSTG-2015} and \cite[Proposition 2.1]{HS-2015} that $\rho(I) = 3/2$ and that $I^{(3)} \not\subseteq I^2$.

It follows by a repeated application of Theorem \ref{thm.rhoa} that $\rhoa(I^{[k]}) = \rhoa(I)$. Furthermore, it follows from Theorem \ref{thm.kfold} that, for all $k \ge 2$,
$$\rho(I^{[k]}) \ge \dfrac{3k}{k+1} > \dfrac{3}{2} = \rho(I).$$
The theorem is proved.
\end{proof}

We end the paper with the following conjectures, that are inspired by what Theorems \ref{thm.kfold} and \ref{thm.example} appear to indicate.

\begin{conjecture} \label{conj.kfold}
There exists a homogeneous polynomial ideal $I$ such that
$$\limsup_{k \rightarrow \infty} \rho(I^{[k]}) = \infty.$$
\end{conjecture}

\begin{conjecture} \label{conj.family}
There exists a family of homogeneous polynomial ideals $\{I_k\}_{k \in \NN}$ (possibly in different polynomial rings) such that
$$\limsup_{k \rightarrow \infty} \left(\rho(I_k) - \rhoa(I_k)\right) = \infty.$$
\end{conjecture}

Note that an affirmative answer to Conjecture \ref{conj.kfold} also gives an affirmative answer to Conjecture \ref{conj.family}.

%%%%%%%%%%%%%%%%%%%%%%%%%%%%%%%%%

\end{document}